\documentclass[3p,final,twocolumn,authoryear]{my-elsarticle}

\usepackage{amsmath,amsthm,mathrsfs,amsfonts}

\journal{Journal of King Saud University--Science}

\usepackage[bookmarksnumbered=true,
bookmarksopen=true,
colorlinks=true,   
pdfborder=001,    
citecolor=red,
linkcolor=cuppink,
anchorcolor=blue,
urlcolor=green]{hyperref}

\newcounter{Theorem}
\theoremstyle{plain}
\newtheorem{theorem}[Theorem]{\bf Theorem}

\newtheorem{corollary}[Theorem]{\bf Corollary}
\newtheorem{lemma}[Theorem]{\bf Lemma}
\theoremstyle{remark}

\theoremstyle{definition}

\newtheorem{example}[Theorem]{\bf Example}

\numberwithin{Theorem}{section} 
\numberwithin{equation}{section}

\DeclareMathOperator{\var}{Var}
\newcommand{\Var}{\mathop{\var}\limits}

\begin{document}
	
\bibliographystyle{abbrvnat}

\begin{frontmatter}

\title{Existence theorems for a nonlinear second-order\\
distributional differential equation\tnoteref{mytitlenote}}

\tnotetext[mytitlenote]{Supported by the program of High-end Foreign
Experts of the SAFEA (No. GDW20163200216) and by FCT and CIDMA 
within project UID/MAT/04106/2013.}


\author[mymainaddress]{Wei Liu}
\ead{liuw626@hhu.edu.cn}

\author[mymainaddress]{Guoju Ye}
\ead{yegj@hhu.edu.cn}

\author[mymainaddress,mysecondaryaddress]{Dafang Zhao}
\ead{dafangzhao@163.com}

\author[mythirdaddress,myfourthaddress]{Delfim F. M. Torres\corref{mycorrespondingauthor}}
\cortext[mycorrespondingauthor]{Corresponding author (delfim@ua.pt).}
\ead{delfim@ua.pt, delfim@aims-cameroon.org}

\address[mymainaddress]{College of Science, Hohai University, Nanjing 210098, P. R. China}

\address[mysecondaryaddress]{School of Mathematics and Statistics, 
Hubei Normal University, Huangshi 435002, P. R. China}

\address[mythirdaddress]{Center for Research and Development 
in Mathematics and Applications (CIDMA),\\ 
Department of Mathematics, University of Aveiro, 
3810-193 Aveiro, Portugal}

\address[myfourthaddress]{African Institute for Mathematical Sciences (AIMS-Cameroon),
P. O. Box 608 Limbe, Cameroon}


\begin{abstract}
In this work, we are concerned with existence of solutions for 
a nonlinear second-order distributional differential equation, 
which contains measure differential equations and stochastic 
differential equations as special cases. The proof is based 
on the Leray--Schauder nonlinear alternative and 
Kurzweil--Henstock--Stieltjes integrals. Meanwhile, 
examples are worked out to demonstrate 
that the main results are sharp.
\end{abstract}



\end{frontmatter}


\section{Introduction}
\label{sec1}

The first-order distributional differential equation (DDE) in the form
\begin{equation}
\label{FODDE}
Dx=f(t,x)+g(t,x)Du,
\end{equation}
where $Dx$ and $Du$ stand, respectively, 
for the distributional derivative of
function $x$ and $u$ in the sense of Schwartz,
has been studied as a perturbed system of the ordinary
differential equation (ODE)
\begin{equation*}
x'=f(t,x)\quad \left(':=\frac{d}{dt}\right).
\end{equation*}
The DDE \eqref{FODDE} provides a good model for many physical processes,
biological neural nets, pulse frequency modulation systems and
automatic control problems \citep{DS71,DS72,SL74}. Particularly, when
$u$ is an absolute continuous function, then \eqref{FODDE} reduces to
an ODE. However, in physical systems, one cannot always expect the
perturbations to be well-behaved. For example, if $u$ is a function
of boundary variation, $Du$ can be identified with a Stieltjes
measure and will have the effect of suddenly changing the state of
the system at the points of discontinuity of $u$, that is, the
system could be controlled by some impulsive force. In this case,
\eqref{FODDE} is also called a measure differential equation (MDE),
see \citep{DS71,DS72,DB09,FM13,FMS12,SL74,MS16,BS14,AS13,AS15}.
Results concerning existence, uniqueness, and stability of solutions,
were obtained in those papers. However, this situation is not the worst, 
because it is well-known that the solutions of a MDE, if exist, 
are still functions of bounded variation. The case when $u$ is 
a continuous function has also been considered in \citep{LYWX12,ZY15}. 
The integral there is understood as a Kurzweil--Henstock integral
\citep{PK06,KUR57,Lee,PT93,SSYE,Eri08,MT94,TVR02,YL15} (or
Kurzweil--Henstock--Stieltjes integral, or distributional
Kurzweil--Henstock integral), which is a generalization of the
Lebesgue integral. Especially, if $u$ denotes a Wiener process 
(or Brownian motion), then \eqref{FODDE} becomes a stochastic
differential equation (SDE), see, for example, \citep{BL11,XM08}. 
In this case, $u$ is continuous but pointwise differentiable nowhere,
and the It\^o integral plays an important role there. As for the
relationship between the Kurzweil--Henstock integral and the It\^o
integral, we refer the interested readers to \citep{BL11,CTT01,TC12}
and references therein.

It is well-known that regulated functions (that is, a function
whose one-side limits exist at every point of its domain) contain
continuous functions and functions of bounded variation as special
cases \citep{DF91}. Therefore, it is natural to consider the
situation when $u$ is a regulated function, see \citep{PT93,MT94}.
Denote by $G[0,1]$ the space of all real regulated functions on
$[0,1]$, endowed with the supremum norm $\|\cdot\|$. Since the DDE
allows both the inputs and outputs of the systems to be
discontinuous, most conventional methods for ODEs are inapplicable, 
and thus the study of DDEs becomes very interesting and important.

The purpose of our paper is to apply the Leray--Schauder nonlinear
alternative and Kurzweil--Henstock--Stieltjes integrals to establish
existence of a solution to the second order DDE of type
\begin{equation}
\label{eq1.1}
-D^2x=f(t,x)+g(t,x)Du,\quad t\in [0,1],
\end{equation}
subject to the three-point boundary condition (cf. \citet{SZ15})
\begin{equation}
\label{eq1.2} 
x(0)=\beta Dx(0),
\quad 
Dx(1)+Dx(\eta)=0,
\end{equation}
where $D^2 x$ stands for the second order distributional derivative
of the real function $x \in G[0,1]$, $u\in G[0,1]$, $\beta$ is a
constant, and $\eta \in [0,1]$. This approach is well-motivated since 
this topic has not yet been addressed in the literature, and by the 
fact that the Kurzweil--Henstock--Stieltjes integral is a powerful 
tool for the study of DDEs. We assume that $f$ and $g$ satisfy 
the following assumptions:
\begin{itemize}
\item[$(H_{{1}})$]  $f(t,x)$ is Kurzweil--Henstock integrable 
with respect to $t$ for all $x\in G[0,1]$;

\item[$(H_{{2}})$] $f(t,x)$ is continuous with respect to $x$ 
for all $t\in [0,1]$;

\item[$(H_{{3}})$] there exist nonnegative Kurzweil--Henstock 
integrable functions $k$ and $h$ such that 
$$
- k\| x\|-h\leq f(\cdot,x)
\leq k\| x\|+h 
\quad \forall x \in B_r,
$$ 
where $B_r=\{x\in G[0,1]\ :\ \|x\|\leq r\}$, $r>0$;

\item[$(H_{{4}})$] $g(t,x)$ is a function with bounded variation 
on $[0,1]$ and $g(0,x)=0$ for all $x\in G[0,1]$;

\item[$(H_{{5}})$] $g(t,x)$ is continuous with respect 
to $x$ for all $t\in [0,1]$;

\item[$(H_{{6}})$] there exists $M>0$ such that 
$$
\sup_{x \in B_r}\Var_{[0,1]}g\leq M,
$$ 
where 
$$
\Var_{[0,1]}g=\sup\sum\limits_n |g(s_n,x(s_n))-g(t_n,x(t_n))|,
$$ 
the supremum taken over every sequence $\{(t_n, s_n)\}$ 
of disjoint intervals in $[0,1]$, 
is called the total variation of $g$ on $[0,1]$.
\end{itemize}

Now, we state our main result.

\begin{theorem}[Existence of a solution to problem \eqref{eq1.1}--\eqref{eq1.2}]
\label{thm3.1}
Suppose assumptions $(H_{1})$--$(H_{6})$ hold. If
\begin{equation*}
(|\beta|+2)\max_{t\in[0,1]}\left|\int_0^t k(s)ds\right|<1,
\end{equation*}
then problem \eqref{eq1.1}--\eqref{eq1.2} has at least one solution.
\end{theorem}

If $k(t)\equiv 0$ on $[0,1]$, then $(H_3)$ can be reduced to
\begin{itemize}
\item[$(H_3')$] there exists a nonnegative 
Kurzweil--Henstock function $h$ such that 
$$
-h\leq f(\cdot,x)\leq h
\quad \forall x \in B_r.
$$
\end{itemize}
Thus, the following result holds as a direct consequence.

\begin{corollary}
\label{Cor3.1}
Assume that $(H_{1})$, $(H_{2})$, $(H_{3}')$ and $(H_{4})$--$(H_{6})$ 
are fulfilled. Then, problem \eqref{eq1.1}--\eqref{eq1.2} 
has at least one solution.
\end{corollary}

It is worth to mention that condition $(H_3')$, together 
with $(H_{1})$ and $(H_{2})$, was firstly proposed by 
\citep{TSC}, to deal with first-order Cauchy problems.

The paper is organized as follows. In Section~\ref{sec2},
we give two useful lemmas: we prove that under our hypotheses
problem \eqref{eq1.1}--\eqref{eq1.2} can be rewritten in an 
equivalent integral form (Lemma~\ref{lem3.1}) and we recall
the Leray--Schauder theorem (Lemma~\ref{lem3.2}).
Then, in Section~\ref{sec3}, we prove our existence result
(Theorem~\ref{thm3.1}). We end with Section~\ref{sec:4}, 
providing two illustrative examples. Along all the manuscript, 
and unless stated otherwise, we always assume that 
$x, u\in G[0,1]$. Moreover, we use the symbol $\int_a^b$ 
to mean $\int_{[a,b]}$.


\section{Auxiliary Lemmas}
\label{sec2}

By $(H_1)$ and $(H_4)$, we define
\begin{equation}
\label{eqFG}
\begin{split}
F(t,x)&=\int_0^t f(s,x(s))ds,\\
G_u(t,x)&=\int_0^t g(s,x(s))du(s),
\end{split}
\end{equation}
for all $t\in[0,1]$.

\begin{lemma}
\label{lem3.1}
Under the assumptions $(H_1)$--$(H_6)$,
problem \eqref{eq1.1}--\eqref{eq1.2} 
is equivalent to the integral equation
\begin{equation}
\label{eq3.1}
\begin{split}
x(t)=&\frac{t+\beta}{2}\left(F(1,x)
+F(\eta,x)+G_u(1,x)+G_u(\eta,x)\right)\\
&-\int_0^t F(s,x)ds-\int_0^t G_u(s,x)ds
\end{split}
\end{equation}
on $[0,1]$, where $F$ and $G_u$ are given in \eqref{eqFG}, 
$u\in G[0,1]$, and $\beta$ and $\eta$ are constants 
with $0\leq \eta\leq 1$.
\end{lemma}

\begin{proof}
For all $t\in [0,1]$, $s\in[0,1]$, and $x\in G[0,1]$, we have
\begin{equation}
\label{eq3.2}
\int_{0}^ts D^2x(s)ds=\int_{0}^ts d(Dx(s))=tDx(t)-x(t)+ x(0)
\end{equation}
by the properties of the distributional derivative.
Integrating \eqref{eq1.1} once over $[0,t]$, we obtain that
\begin{equation}
\label{eq3.4}
Dx(t)=Dx(0)-F(t,x)- G_u(t,x).
\end{equation}
Combining with the boundary conditions \eqref{eq1.2}, one has
\begin{equation}
\label{eq3.5}
Dx(0)=\frac{1}{2}\left(
F(1,x)+F(\eta,x)+G_u(1,x)+G_u(\eta,x)\right)
\end{equation}
and
\begin{equation}
\label{eq3.6}
x(0)=\frac{\beta}{2}\left(
F(1,x)+F(\eta,x)+G_u(1,x)+G_u(\eta,x)\right).
\end{equation}
It follows from \eqref{eq3.2} and \eqref{eq3.4} that
\begin{equation}
\label{eq3.7}
\begin{split}
x(t)=&tDx(0)+x(0)-\int_{0}^t (t-s)f(s,x(s))ds\\
&-\int_{0}^t (t-s)g(s,x(s))du(s).
\end{split}
\end{equation}
Therefore, by \eqref{eq3.5}--\eqref{eq3.7} and the 
substitution formula \citep[Theorem 2.3.19]{TVR02}, one has
\begin{equation*}
\begin{split}
x(t)= &\frac{t+\beta}{2}\left(
F(1,x)+F(\eta,x)+G_u(1,x)+G_u(\eta,x)\right)\\
&-\int_0^t F(s,x)ds-\int_0^t G_u(s,x)ds, 
\quad t\in [0,1].
\end{split}
\end{equation*}
It is not difficult to calculate that 
\eqref{eq1.1}--\eqref{eq1.2} holds by taking the derivative both
sides of \eqref{eq3.1}. This completes the proof.
\end{proof}

Now, we present the well-known Leray--Schauder 
nonlinear alternative theorem.

\begin{lemma}[See \citet{KD85}]
\label{lem3.2}
Let $E$ be a Banach space, $\Omega$ a bounded open subset of $E$, 
$0\in \Omega$, and $T: \overline {\Omega}\rightarrow E$ be a 
completely continuous operator. Then, either there exists 
$x\in \partial \Omega$ such that $T(x)=\lambda x$ with $\lambda>1$, 
or there exists a fixed point $x^*\in\overline{\Omega}$.
\end{lemma}

We prove existence of a solution to problem 
\eqref{eq1.1}--\eqref{eq1.2} with the help 
of the preceding two lemmas.


\section{Proof of Theorem~\ref{thm3.1}}
\label{sec3}

Let
\begin{equation}
\label{eq3.9}
\begin{split}
H(t)&=\int_0^t h(s)ds,\\
K(t)&=\int_0^t k(s)ds, 
\end{split}
\end{equation}
$t\in [0,1]$. Then, by $(H_3)$, 
$H$ and $K$ are continuous functions.
According to \eqref{eqFG} and ($H_1$), function $F$ is
continuous on [0,1], and
\begin{equation*}
\|F\|=\max_{t\in[0,1]}\left|\int_0^t f(s,x(s))ds\right|
\leq \|K\|\|x\|+\|H\|.
\end{equation*}
On the other hand, by \citep[Proposition 2.3.16]{TVR02} and $(H_4)$,
$G_u$ is regulated on $[0,1]$. Further, from $(H_6)$ and the
H\"older inequality \citep[Theorem 2.3.8]{TVR02}, it follows that
\begin{equation*}
\begin{split}
\|G_u\|&\leq \left(|g(0,x(0))|+ |g(1,x(1))|+ \Var_{[0,1]} g
\right)\|u\|\\
&\leq 2M\|u\|.
\end{split}
\end{equation*}
Let
\begin{equation}
\label{eq3.14}
r=\frac{(|\beta|+2)(\|H\|+2M\|u\|)}{1-(|\beta|+2)\|K\|}>0.
\end{equation}
For each $x\in B_r$ and $t\in [0,1]$, define the operator 
$\mathcal T:G[0,1]\to G[0,1]$ by 
\begin{equation}
\label{eq3.15}
\begin{split}
\mathcal{T}x(t):=&\frac{t+\beta}{2}\left(
F(1,x)+F(\eta,x)+G_u(1,x)+G_u(\eta,x)\right)\\
&-\int_0^t F(s,x)ds-\int_0^t G_u(s,x)ds.
\end{split}
\end{equation}
We prove that $\mathcal{T}$ is completely continuous in three steps.
Step 1: we show that $\mathcal{T}:B_r\rightarrow B_r$. Indeed, 
for all $x\in B_r$, one has
\begin{equation}
\label{eq3.16}
\begin{split}
\|\mathcal{T}x\| &\leq  (|\beta|+2)(\|F\|+\|G_u\|)\\ 
&\leq (|\beta|+2)(r\|K\|+\|H\|+2M\|u\|)\\
&= r
\end{split}
\end{equation}
by \eqref{eq3.14} and \eqref{eq3.15}.
Hence, $\mathcal{T}(B_r)\subseteq B_r$.
Step 2: we show that $\mathcal{T}(B_r)$ is equiregulated 
(see the definition in \citet{DF91}). For $t_{0}\in[0, 1)$ 
and $x\in B_r$, we have
\begin{equation*}
\begin{split}
&\left|\mathcal{T}x(t)-\mathcal{T}x(t_{0+})\right|\\
&=\left|\frac{t-(t_{0+})}{2}(F(1,x)+F(\eta,x)+G_u(1,x)+G_u(\eta,x))\right.\\
&\left. \qquad -\int_{t_{0+}}^{t}F(s,x)+G_u(s,x)ds\right|\\
&\leq 2\left|t-(t_{0+})\right|\left(r\|K\|+\|H\|+2M\|u\|\right)
\longrightarrow 0 
\end{split}
\end{equation*}
as $t\to t_{0+}$. Similarly, we can prove that 
$$
\left|\mathcal T x(t_{0-})
-\mathcal Tx(t)\right|\to 0 \text{ as } t\to t_{0-}
$$ 
for each $t_{0}\in(0,1]$. Therefore, $\mathcal{T}(B_r)$ 
is equiregulated on $[0,1]$. In view of Steps 1 and 2 
and an Ascoli--Arzel\`{a} type theorem
\citep[Corollary 2.4]{DF91}, we conclude that $\mathcal{T}(B_r)$ 
is relatively compact. Step 3: we prove that $\mathcal{T}$ 
is a continuous mapping. Let $x\in B_r$ and $\{x_n\}_{n\in \mathbb{N}}$ 
be a sequence in $B_r$ with $x_n \to x$ as $n\to \infty$. 
By $(H_{2})$ and $(H_{4})$, one has
$$
f(\cdot,x_n)\rightarrow f(\cdot,x)
\quad \text{and} \quad g(\cdot,x_n)\rightarrow g(\cdot,x)
$$ 
as $n\rightarrow \infty$.
According to the assumption $(H_{3})$ and
the convergence Theorem 4.3 of \citep{Lee}, we have
$$
\lim _{n \to \infty}  \int_{0}^t f(s,x_n(s))ds 
= \int_0^t f(s,x(s))ds, 
\quad t\in [0,1].
$$ 
Moreover, $(H_{6})$, together with the convergence 
Theorem~1.7 of \citep{PK06}, yields that
$$
\lim _{n \to \infty} \int_{0}^t g(s,x_n(s))du(s) 
= \int_0^t g(s,x(s))du(s), 
$$
$t\in [0,1]$. Hence,
\begin{equation*}
\begin{split}
&\mathcal{T}x_{n}(t)-\mathcal{T}x(t)\\
&= \frac{\beta+t}{2}\left[\left(F(1,x_n)
+F(\eta,x_n)+G_u(1,x_n)+G_u(\eta,x_n)\right)\right.\\
&\left.\quad -(F(1,x)+F(\eta,x)+G_u(1,x)+G_u(\eta,x))\right]\\
&\quad -\int_{0}^t F(s,x_n(s))-F(s,x(s))ds\\
&\quad -\int_{0}^t G_u(s,x_n)-G_u(s,x)ds,\quad t\in[0,1].
\end{split}
\end{equation*}
Therefore, $\lim_{n\to \infty}\mathcal{T}x_n(\cdot)=\mathcal{T}x(\cdot)$,
and thus $\mathcal{T}$ is a completely continuous operator.
Finally, let 
$$
\Omega=\left\{x\in G[0,1]\ :\ \|x\|<r\right\}
$$ 
and assume that $x\in\partial \Omega$ 
such that $\mathcal{T}x=\lambda x$ for $\lambda>1$.
Then, by \eqref{eq3.16}, one has
\begin{equation*}
\lambda r =\lambda \|x\|=\|\mathcal{T}x\|\leq r,
\end{equation*}
which implies that $\lambda \leq 1$. This is a contradiction.
Therefore, by Lemma~\ref{lem3.2}, there exists a fixed
point of $\mathcal{T}$, which is a solution 
of problem \eqref{eq1.1}$-$\eqref{eq1.2}.
The proof of Theorem~\ref{thm3.1} is complete.


\section{Illustrative Examples}
\label{sec:4}

We now give two examples to illustrate Theorem~\ref{thm3.1} 
and Corollary~\ref{Cor3.1}, respectively. Let $g^*(t,x(t))=0$ 
if $t=0$ and $g^*(t,x(t))=1$ if $t\in (0,1]$ for all $x\in B_r$. 
Then, it is easy to see that $g^*$ satisfies hypotheses 
$(H_4)$--$(H_6)$ with $M=1$.

\begin{example}
\label{exa4.1}
Consider the boundary value problem
\begin{equation}
\label{eq4.1}
\left\{
\begin{array}{lcr}
-D^2 x=\frac{x\sin(x)}{3\sqrt{5+t}}+g^*(t,x)D\mathcal
H(t-\frac{1}{2}) ,\ \ t\in [0,1],\\
x(0)=4Dx(0),\quad Dx(1)+Dx(\frac{1}{4})=0,
\end{array}\right.
\end{equation}
where $\mathcal H$ is the Heaviside function, i.e., 
$\mathcal H(t)=0$ if $t<0$ and $\mathcal H(t)=1$ if $t>0$.
It is easy to see that $\mathcal H$ is of bounded variation, 
but not continuous. Let $f(t,x)=\frac{x\sin(x)}{2\sqrt{4+t}}$,
$g(t,x)=g^*(t,x)$, and $u(t)=\mathcal H(t-\frac{1}{2})$. Then,
$(H_1)$, $(H_2)$, and $(H_4)$--$(H_6)$ hold. Moreover, there exist $HK$
integrable functions $k(t)=\frac{1}{3\sqrt{5+t}}$ and $h(t)=1$ such that
$$
-k\|x\|-h\leq  f(\cdot,x)\leq k\|x\|+h
\quad \forall x\in G[0,1],
$$
i.e., $(H_3)$ holds. Further, by \eqref{eq3.9},
$$
\|K\|=\frac{2}{3}\left(\sqrt{6}-\sqrt{5}\right),
\quad \|H\|=1,
\quad \|u\|=\|\mathcal H\|=1.
$$ 
Let $\beta=4$ and $\eta=\frac{1}{4}$. 
From \eqref{eq3.14}, we have
\begin{equation*}
r=\frac{(|\beta|+2)(\|H\|+2M\|u\|)}{1-(|\beta|+2)\|K\|}
=\frac{18}{1-4(\sqrt{6}-\sqrt{5})}.
\end{equation*}
Therefore, by Theorem~\ref{thm3.1}, problem \eqref{eq4.1} 
has at least one solution $x^*$ with
$$
\|x^*\| \leq \frac{18}{1-4(\sqrt{6}-\sqrt{5})}.
$$
\end{example}

\begin{example}
\label{exa4.2}
Consider the boundary value problem
\begin{equation}
\label{eq4.4}
\left \{
\begin{array}{lcr}
-D^2x=\sin(x)+ 2t\sin( {t}^{-2})-\frac {2}{t}\cos({t}^{-2})\\
\qquad +g^*(t,x)D\mathcal W,
\quad t\in [0,1],\\
x(0)=-\frac{1}{6}Dx(0),\quad Dx(1)+Dx(\frac{2}{3})=0,
\end{array}\right.
\end{equation}
where $\mathcal W$ is the Weierstrass function 
$$
\mathcal W(t)=\sum_{n=1}^{\infty}\frac{\sin 7^n\pi t}{2^n}
$$ 
in \citet{GHH16}. It is well-known that $\mathcal W(t)$ 
is continuous but pointwise differentiable nowhere on $[0,1]$, 
so $\mathcal W(t)$ is not of bounded variation. Let
\begin{equation*}
\begin{split}
f(t,x)&=\sin(x)+ 2t\sin( {t}^{-2}) -\frac{2}{t}\cos({t}^{-2}),\\
g(t,x)&=g^*(t,x),\\ 
u&=\mathcal  W. 
\end{split}
\end{equation*}
Then, $(H_1)$, $(H_2)$ and $(H_4)$--$(H_6)$ hold. 
Moreover, let
\begin{equation*}
k(t)=0,\quad h(t)=1+2t\sin( {t}^{-2})
-\frac {2}{t}\cos({t}^{-2}).
\end{equation*}
Obviously, the highly oscillating function $h(t)$ is
Kurzweil--Henstock integrable but not Lebesgue integrable, and
\begin{equation*}
H(t)=\int_0^t h(s)ds= \left\{\begin{array}{ll}
t+t^2\sin(t^{-2}), \quad &t\in(0,1],\\
0, &t=0.
\end{array}\right.
\end{equation*}
Moreover, we have
$$
-h\leq  f(\cdot,x)\leq h \quad \forall x\in G[0,1],
$$
that is, $(H_3)$ holds. Let $\beta=-\frac{1}{6}$ 
and $\eta=\frac{2}{3}$.	Since 
$$
0\leq \|u\|=\|\mathcal W\|
\leq \sum_{n=1}^{\infty}\frac{1}{2^n}=1,
\quad \|H\|_{\infty}=1+\sin(1),
$$
we have by \eqref{eq3.14} that
\begin{equation*}
\begin{split}
3.9899
&\approx\frac{13}{6}(\sin(1)+1)\\
&\leq  r = \frac{(|\beta|+2)(\|H\|+2M\|u\|)}{1-(|\beta|+2)\|K\|_{\infty}}\\
&\leq \frac{13}{6}(\sin(1)+3)\\
&\approx 8.3232.
\end{split}
\end{equation*}
Therefore, by Corollary~\ref{Cor3.1}, problem \eqref{eq4.4} 
has at least one solution $x^*$ with 
$$
\|x^*\|\leq
\frac{13}{6}(\sin(1)+3).
$$
\end{example}


\section*{Acknowledgments}

The authors are grateful to two referees for their
comments and suggestions. 



\end{document}